\documentclass[reqno,11pt]{amsart}
\usepackage[foot]{amsaddr}
\usepackage{amssymb}
\usepackage{bbold}
\usepackage[font=small]{caption}
\usepackage{charter}
\usepackage{enumitem}
\usepackage{floatrow}
\usepackage[margin=1in]{geometry}
\usepackage[colorlinks=true,linktoc=all,linkcolor=blue,citecolor=blue]{hyperref}
\usepackage{mathrsfs}
\usepackage{subcaption}
\usepackage[many]{tcolorbox}
\usepackage{tikz}

\theoremstyle{plain}
\newtheorem{theorem}{Theorem}[section]
\newtheorem*{theorem*}{Theorem}

\newtheorem{lemma}[theorem]{Lemma}
\newtheorem{corollary}[theorem]{Corollary}
\theoremstyle{definition}

\newtheorem*{definition*}{Definition}

\newtheorem*{conjecture*}{Conjecture}
\numberwithin{equation}{section}
\everymath{\displaystyle}
\allowdisplaybreaks

\newcommand{\C}{\mathbb{C}}
\newcommand{\D}{\mathbb{D}}
\newcommand{\N}{\mathbb{N}}
\newcommand{\R}{\mathbb{R}}

\newcommand{\Z}{\mathbb{Z}}
\newcommand{\U}{\mathcal{U}}
\newcommand{\BB}{\mathscr{B}}
\newcommand{\CC}{\mathscr{C}}

\newcommand{\CK}{\mathscr{C}_K}
\newcommand{\Comp}{\mathscr{C}(\Linf)}
\newcommand{\bigchi}{\mbox{\large$\chi$}}
\newcommand{\Linf}{L^{\hspace{-.2ex}\infty}}
\DeclareMathOperator{\metric}{d}
\newcommand{\x}{x}
\newcommand{\y}{y}
\newcommand{\z}{z}
\newcommand{\Sfin}{S_F}
\newcommand{\Sinf}{S_I}
\newcommand{\TI}{\Bumpeq}
\newcommand{\TII}{\bumpeq}
\newcommand{\TIIIa}{\Doteq}
\newcommand{\TIIIb}{\risingdotseq}

\newtcolorbox{compbox}{sidebyside, sidebyside align=top, colback=white, boxrule=.5pt, fontupper=\footnotesize, fontlower=\footnotesize}

\newlist{alphaList}{enumerate}{1} 
\setlist[alphaList]{label=\normalfont{(\alph*)},ref=(\alph*)}
\newlist{numList}{enumerate}{1}
\setlist[numList]{label=\normalfont{\arabic*.}}

\title[Topological Structure of $\Comp$]{Topological Structure of the Space of Composition Operators on $\Linf$ of an Unbounded, Locally Finite Metric Space}

\author{Robert F. Allen\textsuperscript{1}, Whitney George\textsuperscript{1}, and Matthew A. Pons\textsuperscript{2}}
\address{\textsuperscript{1}Department of Mathematics and Statistics, University of Wisconsin-La Crosse}
\address{\textsuperscript{2}Department of Mathematics and Actuarial Science, North Central College}
\email{rallen@uwlax.edu, wgeorge@uwlax.edu, mapons@noctrl}


\begin{document}

\begin{abstract}
We study properties of the topological space of composition operators on the Banach algebra of bounded functions on an unbounded, locally finite metric space in the operator norm topology and essential norm topology. Moreover, we characterize the compactness of differences of two such composition operators.
\end{abstract}

\maketitle

\section{Introduction}
Let $\mathcal{X}$ be a Banach space of functions on a domain $\Omega$, and $S(\Omega)$ the set of self-maps of $\Omega$.  For $\varphi$ in $S(\Omega)$, the induced linear operator $C_\varphi:\mathcal{X}\to\mathcal{X}$, defined by \[C_\varphi f = f\circ\varphi;\quad f \in \mathcal{X},\] is called the \textit{composition operator with symbol $\varphi$}. 

This operator was first studied by Nordgren in \cite{Nordgren:1968} where $\Omega$ is taken to be the open unit disk $\D = \{z \in \C : |z| < 1\}$ of $\C$ and $\mathcal{X}$ is the space $H^2(\D)$ defined by 
\[H^2(\D) = \left\{f:\D\to\C \text{ analytic} \;\Bigr\lvert\; \lim_{r \to 1^-} \int_0^{2\pi} \left|f(re^{i\theta})\right|^2\frac{d\theta}{2\pi} < \infty\right\},\] where $d\theta$ is Lebesgue arc-length measure on $\partial\D$.  This space is known as the Hardy-Hilbert space, as it is a Hilbert space within the family of the Hardy spaces $H^p(\D)$, defined for $1 \leq p < \infty$ as
\[H^p(\D) = \left\{f:\D\to\C \text{ analytic} \;\Bigr\lvert\; \lim_{r \to 1^-} \int_0^{2\pi} \left|f(re^{i\theta})\right|^p\frac{d\theta}{2\pi} < \infty\right\}.\]
Typical in the study of composition operators, $H^2(\D)$ and $H^p(\D)$ are the initial spaces of inquiry.  For further information on composition operators acting on $H^2(\D)$ specifically, or more analytic function spaces, the reader is directed to \cite{Shapiro:1993} and \cite{CowenMacCluer:1995}, respectively.

The Banach space of bounded linear operators on $\mathcal{X}$ is denoted by $\BB(\mathcal{X})$.  The operator norm induces a metric space structure on $\BB(\mathcal{X})$ called the operator (or uniform) norm topology.  As a metric space, we can view operators as points in a topological space.  Thus, topological questions can be considered on spaces of operators.  In this paper, we will study the topological structure of the subset of bounded composition operators within $\BB(\mathcal{X})$, denoted by $\CC(\mathcal{X})$.

The question of identifying isolated points in $\CC(\mathcal{X})$ with the operator norm topology was first considered by Berkson and Porta in \cite{BerksonPorta:1980}. It was shown that the identity composition operator acting on $H^p(\D)$ is isolated in $\CC(H^p)$ with the operator norm topology.  In \cite{Berkson:1981}, Berkson generalized this to show if $\varphi$ is a self-map of $\D$ with radial limit function satisfying $|\varphi(\zeta)| = 1$ for all $\zeta$ in a subset $E \subseteq \partial\D$ of positive measure, then $C_\varphi$ is isolated in $\CC(H^p)$.

Afterward, Shapiro and Sundberg in \cite{ShapiroSundberg:1990} investigated whether these self-maps were the only ones to induce isolated composition operators in $\CC(H^2)$ with the operator norm topology.  They showed this was not the case by constructing an analytic self-map $\varphi$ of $\D$ whose radial limit function satisfies $|\varphi(\zeta)| < 1$ almost everywhere on $\partial\D$ that induces an isolated composition operator in $\CC(H^2)$.  They were not able to completely characterize the symbols that induce isolated points in this topology, and presented the following three problems:
\begin{numList}
\item Characterize the components of $\CC(H^2)$.
\item Which composition operators are isolated in $\CC(H^2)$?
\item Which composition differences are compact on $H^2$?
\end{numList}
In addition, Shapiro and Sundberg offered the following conjecture.

\begin{conjecture*}[Shapiro-Sundberg Conjecture] The set of all composition operators that differ from the given one by a compact operator forms a component in $\CC(H^2)$ with the operator norm topology.
\end{conjecture*}

While the solutions to these problems still elude mathematicians in the $H^2(\D)$ setting, much work has been done to solve the problems for other spaces $\mathcal{X}$ of analytic functions on $\D$.  Manhas provides a survey in \cite{Manhas:2007} of the work on topological structures of composition operators and weighted composition operators on various spaces including the Bergman space, Dirichlet space, Bloch space and weighted Banach space of analytic functions.  The topological structure of (weighted) composition operators between different spaces has also been studied (see \cite{Goebeler:2001,IzuchiOhno:2014,KhoiThomTien:2021}).  This work has migrated to Banach spaces of holomorphic functions in several complex variables as well; see \cite{Towes:2004,HammondMacCluer:2005} for examples of such.

MacCluer, Ohno, and Zhao studied the topological structure of composition operators, with the operator norm topology, acting on $H^\infty(\D)$, the space of bounded analytic functions on $\D$ in \cite{MacCluerOhnoZhao:2001}.  Specifically, they provide geometric characterizations of the compact composition differences, the isolated composition operators, and when two composition operators are in the same path component.  They pose the question of whether isolated composition operators in the operator norm topology are isolated in the essential norm topology.  Hosokawa, Izuchi, and Zheng answer this question in the affirmative; the isolated composition operators on $H^\infty(\D)$ are precisely the essentially isolated composition operators \cite{HosokawaIzuchiZheng:2002}.

In this paper, we wish to begin the study of topological structure of operators acting on discrete function spaces, that is function spaces defined on discrete structures.  In \cite{ColonnaEasley:2012}, Colonna and Easley defined the space of Lipschitz functions $\mathcal{L}(T)$ and the space of bounded functions $\Linf(T)$ on a tree $T$, and studied the multiplication operator between them.  Since then, several discrete function spaces have been defined, and multiplication, composition, and weighted composition operators have been studied (see \cite{AllenJackson:2022} and \cite{AllenPons:2022}, and the references therein).  In general, the definitions of these discrete spaces utilize norms that have the same form as their classical counterparts.  Because of this, certain operator-theoretic results have similar forms as well.  It is our hope that this line of research can expose structural similarities and/or differences between the discrete and classical spaces.

\subsection{Organization of the Paper}
In Section \ref{Section:Preliminary}, we collect the necessary information about unbounded, locally finite metric spaces $T$, and the space of bounded functions $\Linf(T)$.  We define various useful relations on the set of self-maps of $T$ and collect several useful facts about composition operators on $\Linf(T)$.

In Section \ref{Section:CompositionDifferences}, we characterize the compact difference $C_\varphi-C_\psi:\Linf(T) \to \Linf(T)$.  In addition, we determine the operator norm and essential norm of $C_\varphi-C_\psi$.  

In Section \ref{Section:Topology}, we study the structure of the set $\Comp$ under the operator norm topology and show that $\Comp$ is totally disconnected in the operator norm topology. Additionally, we study $\Comp$ under the essential norm topology.  We completely determine the essential components and path components, and show no composition operator is essentially isolated.  We summarize the results of this section below.

\begin{theorem*} Every composition operator in $\Comp$ is isolated in the operator norm topology.
\end{theorem*}

\begin{theorem*}
Let $\varphi$ and $\psi$ be self-maps of $T$.  Then the following are equivalent:
\begin{alphaList}
\item $C_\varphi$ and $C_\psi$ are in the same essential component,
\item $C_\varphi$ and $C_\psi$ are in the same essential path component,
\item $C_\varphi-C_\psi$ is compact.
\end{alphaList}
\end{theorem*}

In the last section of this paper, we compare the topological structures of $\Comp$ and $\CC(H^\infty)$.  While the norms on the spaces $\Linf(T)$ and $H^\infty(\D)$ have the same form, and many operator-theoretic properties are analogous, we seek to determine if there are structural similarities between the two spaces, as viewed through the lens of composition operators.

\section{Bounded functions on unbounded, locally finite metric spaces}\label{Section:Preliminary}
The domains of the functions in this paper are metric spaces that are locally finite, with a distinguished element $o$, called the root.  Recall, a metric space $(T,\metric)$ is locally finite if for every $M>0$, the set $\{v \in T : \metric(o,v) < M\}$ is finite.  For a point $v$ in $T$, we define the length of $v$ by $|v| = \metric(o,v)$.    As the length of a point is used throughout, and not specifically the metric $\metric$, we will denote the metric space simply by $T$. We denote by $T^*$ the set $T-\{o\}$.  In this paper, we assume the locally finite metric space $T$ has root $o$ and is unbounded, that is for every $M > 0$, there exists $v \in T$ with $|v| \geq M$.  Note the condition that $T$ is unbounded implies $T$ admits strictly increasing sequences, by which we mean a sequence of points $(v_n)$ of $T$ for which $(|v_n|)$ is strictly increasing in $\Z_{\geq 0}$.  

By a function on $T$, we mean a complex-valued function $f:T \to \C$.  We denote the set of all self-maps of $T$ by $S(T)$. We say $f$ in $S(T)$ has finite range if there exists a constant $M > 0$ such that $|f(v)| \leq M$ for all $v$ in $T$; otherwise $f$ is said to have infinite range.  We denote the set of self-maps of $T$ with finite range by $\Sfin(T)$ and the set of self-maps with infinite range by $\Sinf(T)$. For $f$ and $g$ in $S(T)$, we define $N_{f,g} = \{v \in T : f(v) \neq g(v)\}$.

We define the following binary relations on $S(T)$:
\[\begin{array}{c@{}ll}
f &{}\TI g & \text{if and only if $f$ and $g$ are elements of $\Sfin(T)$,}\\[.25em]
f &{}\TII g & \text{if and only if exactly one function is an element of $\Sfin(T)$,}\\[.25em]
f &{}\TIIIa g & \text{if and only if $f$ and $g$ are elements of $\Sinf(T)$ and $N_{f,g}$ is finite},\\[.25em]
f &{}\TIIIb g & \text{if and only if $f$ and $g$ are elements of $\Sinf(T)$ and $N_{f,g}$ is infinite}.
\end{array}\]
It is straightforward to prove that $\TI$ and $\TIIIa$ are equivalence relations on $\Sfin(T)$ and $\Sinf(T)$, respectively.  Also straightforward is any two functions $f$ and $g$ in $S(T)$ are related by exactly one of $\TII, \TI, \TIIIa, \text{and} \TIIIb$.  A fact that will be useful in later sections is the following.

\begin{lemma}\label{Lemma:ClassesNonTrivial}
Let $f$ be a function in $\Sfin(T)$ and $g$ a function in $\Sinf(T)$.  Then there exist functions $f_1$ and $f_2$ distinct from $f$, and $g_1$ and $g_2$ distinct from $g$ for which $f \TI f_1$, $f \TII f_2$, $g \TIIIa g_1$, and $g \TIIIb g_2$.
\end{lemma}

\begin{proof}
If $f$ is the function that maps all points in $T$ to $o$, let $w$ be a point in $T^*$ and define $f_1(v) = w$ for all $v$ in $T$.  Otherwise, define $f_1(v) = o$ for all $v$ in $T$.  In either case, $f \neq f_1$ and $f \TI f_1$.  Since $f$ has finite range and $g$ has infinite range, we take $f_2 = g$.  Then $f \neq f_2$ and $f \TII f_2$.

We now construct a function $g_1$ of infinite range such that $g \neq g_1$ and $N_{g,g_1}$ is finite.  Let $w$ be a point in $T - \{g(o)\}$ and define $g_1$ on $T$ by \[g_1(v) = \begin{cases}w & \text{if $v = o$,}\\g(v) & \text{otherwise.}\end{cases}\]  Thus $g = g_1$ on $T*$, and so $g \TIIIa g_1$.

Lastly, we construct a function $g_2$ with infinite range for which $N_{g,g_2}$ is infinite.  There exists a strictly increasing sequence $(v_n)$ in $T$ for which $(g(v_n))$ is a strictly increasing sequence in $g(T)$ since $g$ has infinite range.  Define $g_2$ on $T$ by \[g_2(v) = \begin{cases}
g(v_{n+1}) & \text{if $v = v_n$ for some $n \in \N$,}\\
g(v) & \text{otherwise.}
\end{cases}\]
As $(g(v_n))$ is a strictly increasing sequence in $T$, $g(v_n) \neq g(v_{n+1})$ for all $n \in \N$.  Thus $N_{g,g_2} = \{v_n : n \in \N\}$, and so $G \TIIIb g_2$.
\end{proof}

The space of bounded functions on $T$, denoted $\Linf(T)$ or simply $\Linf$, was defined in \cite{ColonnaEasley:2012} as 
\[\Linf(T) = \left\{f:T \to \C \;\Bigr\lvert\; \sup_{v \in T} |f(v)| < \infty\right\}\]
when $T$ is an infinite rooted tree and it was shown in \cite{AllenCraig:2015} that the space $\Linf(T)$ endowed with the norm \[\|f\|_\infty = \sup_{v \in T} |f(v)|\] is a functional Banach space. The proof of \cite{AllenCraig:2015} carries forward for a locally finite metric space $T$, of which an infinite rooted tree is a specific type.

The characteristic functions in $\Linf(T)$ form a rich and important family in the unit ball (see the discussion after Lemma 2.3 in \cite{AllenCraig:2015}).

\begin{lemma}
Let $w$ be a point in $T$, and define the function $\bigchi_w:T \to \C$ by \[\bigchi_w(v) = \begin{cases}1 & \text{if $v = w$},\\0 & \text{otherwise.}\end{cases}\] Then $\|\bigchi_w\|_\infty = 1$.
\end{lemma}

Composition operators acting on $\Linf$ were studied in \cite{AllenPons:2018} and \cite{AllenPons:2022} by the first and last author.  It was shown in \cite{AllenPons:2018}, again on an infinite rooted tree $T$, that every self-map $\varphi$ induces a bounded composition operator $C_\varphi$ on $\Linf$.  The proofs from \cite{AllenPons:2018} also carry forward for a locally finite metric space $T$, and results relevant to this paper are summarized in Theorem \ref{Theorem:CphiResults}.  Recall for a bounded linear operator $A$ acting on Banach space $\mathcal{X}$, the essential norm of $A$ is defined as the distance of $A$ from the compact operators.  Formally, \[\|A\|_e = \inf\left\{\|A - K\| : K \text{ is compact on $\mathcal{X}$}\right\}.\]  As the zero operator is compact, the following relation between the operator norm and essential norm is immediate
\begin{equation}\label{Inequality:EssentialNormNorm}
\|A\|_e \leq \|A\|.
\end{equation}

\begin{theorem}[{\cite[Corollary 3.2 and 4.4]{AllenPons:2018}}]\label{Theorem:CphiResults}
Let $\varphi$ be a self-map of $T$.  Then
\begin{alphaList}
\item\label{CphiBounded} $C_\varphi$ is bounded on $\Linf$ and $\|C_\varphi\|=1$.
\item\label{CphiCompact} $C_\varphi$ is compact on $\Linf$ if and only if $\varphi$ is in $S_F(T)$.  Moreover, \[\|C_\varphi\|_e = \begin{cases}1 & \text{if $\varphi$ in $S_I(T)$},\\0 & \text{otherwise.}\end{cases}\]
\end{alphaList}
\end{theorem}

We denote the set of bounded composition operators on $\Linf$ by $\Comp$ and the subset of compact composition operators by $\CK$. In this paper, we study the topological structure of $\Comp$ under the operator norm topology and the essential norm topology.  To this end, we define the following classes of composition operators, which follow closely the relations defined above on self-maps of $T$.  If $\mathcal{R} \in \{\TI,\TII,\TIIIa,\TIIIb\}$ and $\varphi$ is a self-map of $T$, then the $\mathcal{R}$-class of $C_\varphi$ is 
\[[C_\varphi]_{\mathcal{R}} = \{C_\psi \in \Comp : \varphi \mathcal{R} \psi\}.\] 

\section{Composition Differences on $\Linf$}\label{Section:CompositionDifferences}
To consider the problems of Shapiro and Sundberg on this discrete function space $\Linf$, we next study the difference of composition operators $C_\varphi-C_\psi$ for self-maps $\varphi$ and $\psi$ of $T$.  Understanding when a composition operator differs from another by a compact operator will be fundamental in the study of the topological structure of $\Comp$ under either the operator norm topology or the essential norm topology.

As the characteristic functions play an important role in the study of multiplication, composition, and weighted composition operators on $\Linf$ (see \cite{AllenCraig:2015},\cite{AllenPons:2018}, and \cite{AllenPons:2022}), the difference of characteristic functions will play an important role in the study of composition differences.

\begin{lemma}\label{Lemma:ChiDifference}
Let $\varphi$ and $\psi$ be distinct self-maps of $T$, and let $w$ be a point in $T$ for which $\varphi(w) \neq \psi(w)$.  Define the function $f_w:T \to \C$ by \[f_w(v) = \bigchi_{\varphi(w)}(v) - \bigchi_{\psi(w)}(v).\]  Then $\|f_w\|_\infty = 1$ and $\|(C_\varphi-C_\psi)f_w\|_\infty = 2$. 
\end{lemma}

\begin{proof}
Since $f_w$ is the difference of two characteristic functions, $|f_w(v)| \leq 1$ for all points $v$ in $T$.  By direct calculation, we have 
\begin{equation}\label{Equality:ChiDifferenceValues}
f_w(\varphi(w)) = 1 \text{ and } f_w(\psi(w)) = -1.
\end{equation} Thus $\|f_w\|_\infty \geq |f_w(\varphi(w))| = 1$.  Therefore $\|f_w\|_\infty = 1$. 
Let $v$ be in $T$. Then \[|((C_\varphi-C_\psi)f_w)(v)| \leq |f_w(\varphi(v))| + |f_w(\psi(v))| \leq 2\] and from \eqref{Equality:ChiDifferenceValues} it follows that $|((C_\varphi-C_\psi)f_w)(w)| = 2$ and thus $\|(C_\varphi-C_\psi)f_w\|_\infty = 2$.
\end{proof}

Since, for any two self-maps $\varphi$ and $\psi$ of $T$, the composition operators $C_\varphi$ and $C_\psi$ are bounded on $\Linf$, it follows that $C_\varphi - C_\psi$ is also bounded.  We will first determine the norm of $C_\varphi-C_\psi$.

\begin{theorem}\label{Theorem:CphiDifferenceNorm}
Let $\varphi$ and $\psi$ be self-maps of $T$ with $\varphi\neq\psi$. Then $\|C_\varphi-C_\psi\| = 2$.
\end{theorem}

\begin{proof}
It follows from Theorem \ref{Theorem:CphiResults}\ref{CphiBounded} that $\|C_\varphi-C_\psi\| \leq \|C_\varphi\|+\|C_\psi\| = 2$. As $\varphi \neq \psi$, there exists a point $w$ in $T$ such that $\varphi(w)\neq\psi(w)$. Define the function $f_w:T \to \C$ by $f_w(v) = \bigchi_{\varphi(w)}(v) - \bigchi_{\psi(w)}(v)$.  It follows from Lemma \ref{Lemma:ChiDifference} that $\|C_\varphi-C_\psi\| \geq \|(C_\varphi-C_\psi)f_w\|_\infty = 2.$ Thus $\|C_\varphi-C_\psi\| = 2$.
\end{proof}

We now determine when the composition difference $C_\varphi-C_\psi$ is compact for self-maps $\varphi$ and $\psi$ of $T$.  To this end, we will utilize a tool common in study of compact operators acting on discrete function spaces such as $\Linf$.  Lemma \ref{Lemma:compactness_characterization} is an adaptation of \cite[Lemma 3.7]{Tjani:2003} and was used to characterize the compact multiplication operators \cite{AllenCraig:2015}, composition operator \cite{AllenPons:2018}, and weighted composition operators \cite{AllenPons:2022} acting on $\Linf$. 

\begin{lemma}\label{Lemma:compactness_characterization}
Let $X$ and $Y$ be Banach spaces of functions on $T$. Suppose that
\begin{enumerate}[label=\normalfont{(\roman*)}]
\item the point evaluation functionals of $X$ are bounded,
\item the closed unit ball of $X$ is a compact subset of $X$ in the topology of uniform convergence on compact sets,
\item $A:X \to Y$ is bounded when $X$ and $Y$ are given the topology of uniform convergence on compact sets.
\end{enumerate}  Then $A$ is a compact operator if and only if given a bounded sequence $(f_n)$ in $X$ such that $(f_n)$ converges to zero pointwise, then the sequence $(A f_n)$ converges to zero in the norm of $Y$.
\end{lemma}

By Theorem \ref{Theorem:CphiResults}\ref{CphiCompact}, if the self-maps $\varphi$ and $\psi$ of $T$ satisfy $\varphi \TI \psi$, then $C_\varphi - C_\psi$ is compact.  Also, if $\varphi = \psi$, then $C_\varphi-C_\psi$ is compact. The next result shows that two self-maps of $T$ satisfying $\varphi \TII \psi$ do not induce a compact composition difference.

\begin{lemma}\label{Lemma:mixed_bounded_maps}
Let $\varphi$ and $\psi$ be self-maps of $T$ with $\varphi\TII \psi$.  Then $C_\varphi - C_\psi$ is not compact on $\Linf$.
\end{lemma}

\begin{proof}
Without loss of generality, suppose $\psi$ has finite range.  Then there exists an increasing sequence $(v_n)$ in $T$ for which the sequence $(\varphi(v_n))$ is an increasing sequence in $\varphi(T)$.  Since $\psi(T)$ is finite, we can take every point $\varphi(v_n)$ to be in $T-\psi(T)$. For each $n \in \N$, define $f_n:T \to \C$ by $f_n(v) = \bigchi_{\varphi(v_n)}(v)$.  Note $\|f_n\|_\infty = 1$ for all $n \in \N$ and $(f_n)$ converges to 0 pointwise on $T$. For each $n \in \N$, $((C_\varphi-C_\psi)f_n)(v_n)= 1.$ It follows that $\lim_{n \to \infty}\|(C_\varphi-C_\psi)f_n\|_\infty \neq 0$, and thus $C_\varphi-C_\psi$ is not compact on $\Linf$ by Lemma \ref{Lemma:compactness_characterization}.
\end{proof}

\begin{theorem}\label{Theorem:CompactDifference}
Let $\varphi$ and $\psi$ be self-maps of $T$.  Then $C_\varphi - C_\psi$ is compact on $\Linf$ if and only if $\varphi\TI\psi$ or $\varphi\TIIIa\psi$.
\end{theorem}

\begin{proof}
Suppose $C_\varphi-C_\psi$ is compact on $\Linf$.  From Lemma \ref{Lemma:mixed_bounded_maps}, it follows that either both $\varphi$ and $\psi$ have finite range or neither do.  If both $\varphi$ and $\psi$ have finite range, then $\varphi\TI\psi$.  Suppose now that $\varphi$ and $\psi$ have infinite range. Assume, for purposes of contradiction, that $N_{\varphi,\psi}$ is infinite.  There exists an increasing sequence $(v_n)$ in $T$ for which $\varphi(v_n) \neq \psi(v_n)$ for all $n \in \N$. For each $n \in \N$, define $f_n(v) = \bigchi_{\varphi(v_n)}(v)$.  Note that $\|f_n\|_\infty = 1$ and $(f_n)$ converges to 0 pointwise on $T$.  From Lemma \ref{Lemma:compactness_characterization} it follows that $\|(C_\varphi-C_\psi)f_n\|_\infty \to 0$ as $n \to \infty$.  However by direct calculation, observe that 
$\|(C_\varphi-C_\psi)f_n\|_\infty \geq 1$ for each $n \in \N$, a contradiction.  Thus $N_{\varphi,\psi}$ is finite and so $\varphi\TIIIa\psi$. 

Conversely, suppose $\varphi\TI\psi$ or $\varphi\TIIIa\psi$.  In the case that $\varphi \TI \psi$, it follows that $C_\varphi-C_\psi$ is compact on $\Linf$ from Theorem \ref{Theorem:CphiResults}\ref{CphiCompact} and \cite[Proposition 4.9]{MacCluer:2009}.  To complete the proof, suppose $\varphi\TIIIa\psi$.  If $\varphi=\psi$, then $C_\varphi-C_\psi$ is compact and so we suppose $\varphi\neq\psi$.  Let $(f_n)$ be a bounded sequence in $\Linf$ converging to 0 pointwise and fix $\varepsilon > 0$.  Note that for all $v \in T-N_{\varphi,\psi}$ we have $((C_\varphi-C_\psi)f_n)(v)= 0.$  Thus \[\|(C_\varphi-C_\psi)f_n\|_\infty = \max_{v \in N_{\varphi,\psi}} |((C_\varphi-C_\psi)f_n)(v)|.\] The fact that $N_{\varphi,\psi}$ is finite implies $\{\varphi(v):v\in N_{\varphi,\psi}\}\cup\{\psi(v):v\in N_{\varphi,\psi}\}$ is also finite. Moreover, since $(f_n)$ converges to 0 pointwise on $T$, $(f_n)$ converges uniformly to 0 on finite subsets of $T$ and thus for $n$ sufficiently large, $|f_n(\varphi(v))| < \varepsilon/2$ and $|f_n(\psi(v))| < \varepsilon/2$ for all $v\in N_{\varphi,\psi}$. It immediately follows that \[\|(C_\varphi-C_\psi)f_n\|_\infty \leq \max_{v \in N_{\varphi,\psi}}\left(|f_n(\varphi(v))|+|f_n(\psi(v))|\right)\\
<\varepsilon.\] Thus $\|(C_\varphi-C_\psi)f_n\|_\infty \to 0$ as $n \to \infty$ and $C_\varphi-C_\psi$ is compact on $\Linf$ by Lemma \ref{Lemma:compactness_characterization}.
\end{proof}

We now determine the essential norm of $C_\varphi-C_\psi$.  This will be of use in the study of the essential norm topology (see Section \ref{Section:Topology}).

\begin{theorem}\label{Theorem:EssentialNormDifference} Let $\varphi$ and $\psi$ be self-maps of $T$.  Then
\[\|C_\varphi-C_\psi\|_e = \begin{cases}
2 & \text{if $\varphi\TIIIb\psi$,}\\
1 & \text{if $\varphi\TII\psi$,}\\
0 & \text{if $\varphi\TI\psi$ or $\varphi\TIIIa\psi$.}
\end{cases}\]
\end{theorem}

\begin{proof}
First, if $\varphi \TI \psi$ or $\varphi \TIIIa \psi$ then $\|C_\varphi-C_\psi\|_e = 0$ by Theorem \ref{Theorem:CompactDifference}.  Next, suppose $\varphi\TII\psi$.  So exactly one of $\varphi$ or $\psi$ has infinite range, and without loss of generality, suppose $\varphi$ has infinite range.  Since $C_\psi$ is compact on $\Linf$ by Theorem \ref{Theorem:CphiResults}\ref{CphiCompact}, we have \[1 = \|C_\varphi\|_e = \left|\|C_\varphi\|_e - \|C_\psi\|_e\right| \leq \|C_\varphi-C_\psi\|_e \leq \|C_\varphi\|_e + \|C_\psi\|_e = \|C_\varphi\|_e = 1.\] Therefore, $\|C_\varphi-C_\psi\|_e = 1$.

Finally, suppose $\varphi\TIIIb\psi$.  Then $\varphi$ and $\psi$ have infinite range and $N_{\varphi,\psi}$ is infinite. From \eqref{Inequality:EssentialNormNorm} and Theorem \ref{Theorem:CompactDifference}, we obtain $\|C_\varphi-C_\psi\|_e \leq \|C_\varphi-C_\psi\| = 2.$  Assume, for purposes of contradiction, that $\|C_\varphi-C_\psi\|_e < 2$.  Then there exists a compact operator $K$ on $\Linf$ and constant $M > 0$ such that $\|(C_\varphi-C_\psi)-K\| < M < 2.$  

Since $N_{\varphi,\psi}$ is infinite, there exists an increasing sequence $(v_n)$ in $T$ with $\varphi(v_n)\neq\psi(v_n)$ for all $n \in \N$.  Define for each $n \in \N$ the function $f_n:T \to \C$ by $f_n(v) = \bigchi_{\varphi(v_n)}(v) - \bigchi_{\psi(v_n)}(v)$.  It follows from Lemma \ref{Lemma:ChiDifference} that $\|f_n\|_\infty = 1$.  As $K$ is compact and $(f_n)$ converges to 0 pointwise on $T$, Lemma \ref{Lemma:compactness_characterization} implies $\|Kf_n\|_\infty$ converges to 0 as $n \to \infty$. So there exists an $N \in \N$ such that $\|Kf_n\|_\infty < 1$ for all $n \geq N$.  Let $n \geq N$, and recall $\|(C_\varphi-C_\psi)f_n\|_{\infty} = 2$ from Lemma \ref{Lemma:ChiDifference}.  Then $\|(C_\varphi-C_\psi)f_n\|_\infty - \|Kf_n\|_\infty > 0$ and 
\[\begin{aligned} M &> \|(C_\varphi-C_\psi)-K\| \geq \|((C_\varphi-C_\psi)-K)f_n\|_\infty\\
&\geq \|(C_\varphi - C_\psi) f_n\|_\infty - \|Kf_n\|_\infty = 2-\|Kf_n\|_\infty.\end{aligned}\]
So $2 = \lim_{n \to \infty} \left(2-\|Kf_n\|_\infty\right) \leq M < 2,$ a contradiction.  It then follows that $\|C_\varphi-C_\psi\|_e = 2$.
\end{proof}

\section{Topological Structure of $\Comp$}\label{Section:Topology}

The operator norm $\|\cdot\|$, as the name suggests, is a norm on $\BB(\Linf)$.  This norm induces the metric $\metric_u:\BB(\Linf)\times\BB(\Linf)\to [0,\infty)$ defined by \[\metric_u(x,y) = \|x-y\|\] for all $x$ and $y$ in $\BB(\Linf)$.  The operator norm topology on $\BB(\Linf)$ is the metric space topology induced by metric $\metric_u$, and we denote $\BB(\Linf)$ with the operator norm topology by $\left(\BB(\Linf),\metric_u\right)$.  The set of bounded composition operators $\Comp$ is a subset of $\BB(\Linf)$.  So the operator norm topology on $\Comp$, denoted $\left(\Comp,\metric_u\right)$, is the subspace topology induced from the topology of $\left(\BB(\Linf),\metric_u\right)$.

It follows immediately from Theorem \ref{Theorem:CphiDifferenceNorm} that the operator norm topology on $\Comp$ is discrete, and thus $(\Comp,\metric_u)$ is totally disconnected. As a result of the total disconnectedness of the operator norm topology, $\Comp$ is not an example of a space that satisfies the Shapiro-Sundberg conjecture.

\begin{corollary}
Every point in $(\Comp,\metric_u)$ is isolated.
\end{corollary}

Following the line of inquiry of \cite{MacCluerOhnoZhao:2001} and \cite{HosokawaIzuchiZheng:2002}, we will consider if the isolated composition operators in the operator norm topology are also essentially isolated, that is, isolated in the essential norm topology.  Since every composition operator is isolated in the operator norm topology, we will phrase this by determining if $\Comp$ is essentially totally disconnected.  If the answer is no, this will mark a significant difference in the respective topologies from that of $\CC(H^\infty)$ (see Section \ref{Section:Comparison}).

Unlike the operator norm, the essential norm $\|\cdot\|_e$ is a semi-norm on $\BB(\Linf)$, as any compact operator has essential norm zero.  In the same way the operator norm induces a metric on $\BB(\Linf)$, the essential norm induces a pseudo-metric $\metric_e:\BB(\Linf)\times\BB(\Linf) \to [0,\infty)$ defined by \[\metric_e(x,y) = \|x-y\|_e.\]  We define an open ball in the pseudo-metric as \[B_e(x,\varepsilon) = \{y \in \BB(\Linf) : \metric_e(x,y) = \|x-y\|_e < \varepsilon\}\] for every $x$ in $\BB(\Linf)$ and $\varepsilon > 0$.  The collection of all open balls forms a basis for the essential norm topology $(\BB(\Linf),\metric_e)$.  This topology is known as the pseudo-metric topology (see \cite{Kelley:1955}).  The essential norm topology on $\Comp$ is the induced topology from $(\BB(\Linf),\metric_e)$.  

We see that the essential norm topology is courser than the operator norm topology.  This follows immediately from \eqref{Inequality:EssentialNormNorm}; given $x$ in $\BB(\Linf)$ and $\varepsilon > 0$ we have \[B_u(x,\varepsilon) \subseteq B_e(x,\varepsilon).\]  

Given the results of Theorems \ref{Theorem:CphiResults}\ref{CphiCompact} and \ref{Theorem:EssentialNormDifference}, we can precisely determine the open balls $B_e$.  Recall the set of compact composition operators in $\Comp$ is denoted by $\CK$. From Theorem \ref{Theorem:CphiResults}\ref{CphiCompact} if $\varphi$ is a self-map of $T$ that has finite range, then $[C_\varphi]_{\TI} = \CK$.  We see from the fact that $\|\cdot\|_e$ is a simple function on $\Comp$, the open balls in the essential norm topology have the following formulation. 

\begin{theorem}\label{Theorem:EssentialBalls}
For elements in $\Comp$,  
\begin{alphaList}
\item if $\x \in \CK$, then \[B_e(\x,\varepsilon) = \begin{cases}
\CK & \text{if $0 < \varepsilon \leq 1$,}\\
\Comp & \text{if $\varepsilon > 1$.}
\end{cases}\]
\item if $\y \in \CK^c$, then \[B_e(\y,\varepsilon) = \begin{cases}
[\y]_{\TIIIa} & \text{if $0 < \varepsilon \leq 1$},\\
[\y]_{\TIIIa} \cup \CK & \text{if $1 < \varepsilon \leq 2$},\\
\Comp & \text{if $\varepsilon > 2$.}
\end{cases}\]
\end{alphaList}
\end{theorem}

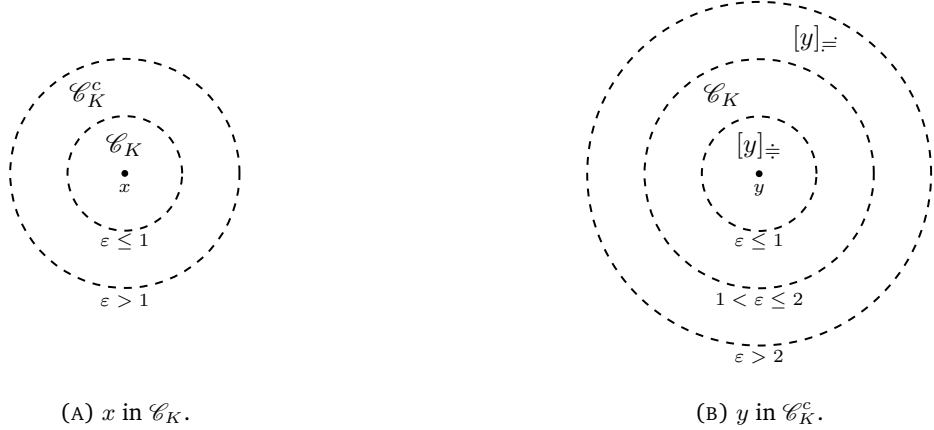
\begin{figure}[htp]
\ffigbox[\FBwidth]
{\begin{subfloatrow}
\ffigbox
{\caption{$\x$ in $\CK$.}}%
{\begin{tikzpicture}[scale=0.75pt]
\coordinate (center);
\draw[line width=0.75pt, dashed] (0,0) circle (.8in);
\draw[line width=0.75pt, dashed] (0,0) circle (.4in);
\filldraw[fill=black] (0,0) circle (.02in);
\node[below] at (0,0) {\tiny $\x$};
\node at (0,-1.2) {\tiny $\varepsilon \leq 1$};
\node at (0,-2.25) {\tiny $\varepsilon > 1$};
\node at (0,0.5) {\small $\CK$};
\node at (-0.65,1.4) {\small $\CK^c$};
\end{tikzpicture}}

\ffigbox
{\caption{$\y$ in $\CK^c$.}}%
{\begin{tikzpicture}[scale=0.75pt]
\coordinate (center);
\draw[line width=0.75pt, dashed] (0,0) circle (1.2in);
\draw[line width=0.75pt, dashed] (0,0) circle (.8in);
\draw[line width=0.75pt, dashed] (0,0) circle (.4in);
\filldraw[fill=black] (0,0) circle (.02in);
\node[below] at (0,0) {\tiny $\y$};
\node at (0,-1.2) {\tiny $\varepsilon \leq 1$};
\node at (0,-2.25) {\tiny $1 < \varepsilon \leq 2$};
\node at (0,-3.25) {\tiny $\varepsilon > 2$};
\node at (0,0.5) {\small $[\y]_{\TIIIa}$};
\node at (-0.65,1.4) {\small $\CK$};
\node at (1,2.4) {\small $[\y]_{\TIIIb}$};
\end{tikzpicture}}
\end{subfloatrow}}
{\caption{Open balls in the essential norm topology.}}
\end{figure}

The following two corollaries are a direct consequence of Theorem \ref{Theorem:EssentialBalls}.  Note for convenience, a neighborhood of a point in the essential norm topology will be referred to as an essential neighborhood.

\begin{corollary}\label{Corollary:EssentialNeighborhoods}
Let $\x$ and $\y$ be in $\Comp$ with $\x \in \CK$ and $\y \in \CK^c$.  Then every essential neighborhood of $\x$ contains $\CK$ and every essential neighborhood of $\y$ contains $[\y]_{\TIIIa}$.
\end{corollary}

\begin{corollary}
The topological space $(\Comp,\metric_e)$ is not Hausdorff.
\end{corollary}

As a consequence of Theorem \ref{Theorem:EssentialBalls}, a point $\x$ in $\CK$ will be essentially isolated if $\CK = \{\x\}$.  Likewise, $\y$ in $\CK^c$ will be essentially isolated if $[\y]_{\TIIIa} = \{\y\}$.  From Lemma \ref{Lemma:ClassesNonTrivial}, neither of these cases are possible.

\begin{corollary}
There are no isolated points in $(\Comp,\metric_e)$.
\end{corollary}

From Theorem \ref{Theorem:EssentialBalls} the set of compact composition operators $\CK$ is open in the essential norm topology, as it corresponds to the open ball $B_e(\x,1/2)$ for any $\x$ in $\CK$.  If $\y$ is in $\CK^c$, then $[\y]_{\TIIIa} = B_e(y,1/2)$ is open as well.  The sets $\CK$ and $[\y]_{\TIIIa}$, for a $\y$ in $\CK^c$, are not only open in the essential norm topology, but are closed as well (see Lemma \ref{Lemma:EssentiallyClosed}).  A topological approach is quite straightforward, whereas an analytic proof would require showing that every convergent sequence of compact composition operators first converges to a composition operator in essential-norm, and then prove that operator must also be compact.  

\begin{lemma}\label{Lemma:EssentiallyClosed}
Let $\y$ be in $\CK^c$.  Then the sets $\CK$ and $[\y]_{\TIIIa}$ are disjoint and closed in the essential norm topology.  
\end{lemma}

\begin{proof}
First, we show $\CK$ and $[\y]_{\TIIIa}$ are disjoint.  Assume there exists $\z$ in $\CK\cap[\y]_{\TIIIa}$.  Then $\z \TIIIa \y$ and $\z \TI x$ for some $x$ in $\CK$.  This is a contradiction since $\TI$ and $\TIIIa$ are distinct equivalence relations.  Thus $\CK \cap [\y]_{\TIIIa} = \varnothing$.

We conclude the proof by showing $\CK$ is closed in the essential norm topology, as the proof for $[\y]_{\TIIIa}$ follows a similar argument.  It follows from Theorem \ref{Theorem:EssentialBalls} and the above argument that $B_e(\y,1/2) = [\y]_{\TIIIa}$ is disjoint from $\CK$. Thus $\CK^c$ is open since $B_e(\y,1/2) \subseteq \CK^c$ for each $\y$ in $\CK^c$. Therefore, $\CK$ is closed in the essential norm topology.
\end{proof}

In fact, not only are $\CK$ and $[\y]_{\TIIIa}$, for any $\y$ in $\CK^c$, closed in the essential norm topology, they are compact.  Note, $(\Comp,\metric_e)$ is not a space for which compact implies closed and so we proved the sets are closed first in Lemma \ref{Lemma:EssentiallyClosed}.  

\begin{lemma}\label{Lemma:EssentiallyCompact}
Let $\y$ be in $\CK^c$.  Then the sets $\CK$ and $[\y]_{\TIIIa}$ are compact in the essential norm topology.
\end{lemma}

\begin{proof}
As with Lemma \ref{Lemma:EssentiallyClosed}, we will prove $\CK$ is compact; the compactness of $[\y]_{\TIIIa}$ follows a similar argument.  Let $\mathscr{O}$ be an open cover of $\CK$ in the essential norm topology.  There exists an open set $\U$ in $\mathscr{O}$ with $\U\cap\CK \neq \varnothing$.  It follows from Corollary \ref{Corollary:EssentialNeighborhoods} that $\CK \subseteq \U$.  Thus $\mathscr{O}$ admits a singleton sub-cover of $\CK$, proving $\CK$ is compact in the essential norm topology. 
\end{proof}

A topological space $X$ is said to be \textit{locally compact} if every point in $X$ has a compact neighborhood.  Also, X is \textit{strongly locally compact}, or \textit{locally relatively compact}, if every point in $X$ has a closed compact neighborhood.  In a Hausdorff space, these properties are equivalent, but in a non-Hausdorff space, such as $(\Comp,\metric_e)$, strongly locally compact implies locally compact.  Lemma \ref{Lemma:EssentiallyClosed} and \ref{Lemma:EssentiallyCompact} show $(\Comp,\metric_e)$ is strongly locally compact.

\begin{corollary}
The space $(\Comp,\metric_e)$ is strongly locally compact.
\end{corollary}

\noindent Since $\TIIIa$ is an equivalence relation on $\CK^c$, the collection $\{\CK\} \cup \{[\y]_{\TIIIa} : \y \in \CK^c\}$ is an open cover of $\Comp$ which does not admit a finite sub-cover.  

\begin{corollary}
The space $(\Comp,\metric_e)$ is not compact.
\end{corollary}

We now consider the essential connectedness and essential components of $\Comp$.  We wish to determine the essentially isolated points in $\Comp$, and to determine if the isolated points and essentially isolated points in $\Comp$ coincide, as they do in $\CC(H^\infty)$.  As $(\Comp,\metric_e)$ contains sets other than $\varnothing$ and $\Comp$ that are both open and closed, it is immediate that $\Comp$ is not connected in the essential norm topology.

\begin{corollary}
The space $(\Comp,\metric_e)$ is not connected.
\end{corollary}

If a topological space $X$ is not connected, then it can be written as a union of connected sets so that for any two such sets, $A$ and $B$, $\overline{A}\cap\overline{B}=\varnothing$.  We say two such sets are \textit{separated}.  These connected sets are called the connected components of $X$. 

\begin{theorem}\label{Theorem:ConnectedBalls}
The connected components of $(\Comp,\metric_e)$ are $\{\CK\} \cup \{[\y]_{\TIIIa} : \y \in \CK^c\}$.
\end{theorem}

\begin{proof}
We begin by showing $\CK$ is a connected subset of $(\Comp,\metric_e)$.   Let $\mathcal{O}_1$ and $\mathcal{O}_2$ be disjoint open subsets of $\CK$ whose union is $\CK$.  Let $\x$ be in $\CK$, and without loss of generality suppose $\x$ is in $\mathcal{O}_1$. There exists $\varepsilon > 0$ such that $B_e(\x,\varepsilon) \subseteq \mathcal{O}_1$.  As $\x\in\CK$, $\CK \subseteq B_e(\x,\varepsilon)$ by Corollary \ref{Corollary:EssentialNeighborhoods}.  This implies $\CK = \mathcal{O}_1$, and so $\mathcal{O}_2 = \varnothing$.  Thus $\CK$ is connected in the essential norm topology.  The proof that $[\y]_{\TIIIa}$, for $\y\in\CK^c$, is connected follows a similar argument.

Let $\y$ in $\CK^c$.  From Lemma \ref{Lemma:EssentiallyClosed}, $\CK$ and $[\y]_{\TIIIa}$ are separated. By Lemma \ref{Lemma:ClassesNonTrivial}, there exists $\z$ in $\CK^c - [y]_{\TIIIa}$. A similar argument as above shows that $[\z]_{\TIIIa}$ and $[\y]_{\TIIIa}$ are separated.  With the observation that \[\Comp = \{\CK\} \cup \bigcup_{\y \in \CK^c} [\y]_{\TIIIa},\] the essential connected components of $\Comp$ are precisely $\CK$ and $[\y]_{\TIIIa}$, for $\y$ in $\CK^c$.
\end{proof}

A topological space $X$ is called \textit{locally connected} if every point of $X$ has a connected neighborhood. Corollary \ref{Corollary:EssentialNeighborhoods} and Theorem \ref{Theorem:ConnectedBalls} show $(\Comp,\metric_e)$ is locally connected.

\begin{corollary}
The space $(\Comp,\metric_e)$ is locally connected.
\end{corollary}

A topological space $X$ is \textit{locally path-connected} if every point of $X$ has a path-connected neighborhood.

\begin{theorem}
The space $(\Comp,\metric_e)$ is locally path-connected.
\end{theorem}

\begin{proof}
We begin by showing $\CK$ is path-connected. Let $\x_1$ and $\x_2$ be in $\CK$ and consider the function $\gamma:[0,1]\to \Comp$ defined by 
\[\gamma(t)=\begin{cases}\x_1,& \text{if $0\leq t<1$,}\\ \x_2,&\text{if $t=1$.}\end{cases}\] Let $\mathcal{U}$ be open in $\Comp$. If $\x_1\in\mathcal{U}$, then $\x_2\in\CK\subseteq \mathcal{U}$ by Corollary \ref{Corollary:EssentialNeighborhoods}. Therefore, $\gamma^{-1}(\mathcal{U})=[0,1]$ which is open. If $\x_1\not\in\mathcal{U}$, then $\U\cap\CK = \varnothing$ also by Corollary \ref{Corollary:EssentialNeighborhoods}.  Thus $\x_2\not\in\mathcal{U}$. Therefore, $\gamma^{-1}(\U)=\varnothing$ which is open. So $\gamma$ is continuous, and thus a path for which $\mathrm{Im}(\gamma)\subseteq \CK$. Since $\x_1$ and $\x_2$ were chosen arbitrarily, the set $\CK$ is path-connected.  A similar argument shows $[\y]_{\TIIIa}$ is path-connected for any $\y$ in $\CK^c$.

Finally, we show $(\Comp,\metric_e)$ is locally path-connected. Let $\z$ be in $\Comp$. Appealing to Corollary \ref{Corollary:EssentialNeighborhoods}, if $\z$ is in $\CK$, then $\CK$ is a path-connected neighborhood of $\z$.  Likewise, if $\z$ is in $\CK^c$, then $[\z]_{\TIIIa}$ is a path-connected neighborhood of $\z$.  Thus every point has a path-connected neighborhood, and so $(\Comp,\metric_e)$ is locally path-connected.   
\end{proof}

In a locally path-connected space, the connected components coincide with the path components.  Thus, we obtain the following result which solve the problems proposed by Shapiro and Sundberg for the space $(\Comp,\metric_e)$.  Additionally, we see that $(\Comp, \metric_e)$ is an example of a space for which the Shapiro-Sundberg conjecture holds.

\begin{theorem}\label{Theorem:EssentialComponentResults}
Let $\varphi$ and $\psi$ be self-maps of $T$.  Then the following are equivalent:
\begin{alphaList}
\item $C_\varphi$ and $C_\psi$ are in the same essential component,
\item $C_\varphi$ and $C_\psi$ are in the same essential path component,
\item $C_\varphi-C_\psi$ is compact,
\item either $\varphi \TI \psi$ or $\varphi \TIIIa \psi$.
\end{alphaList}
\end{theorem}

\noindent We summarize the essential norm topological results of $\Comp$ below.

\begin{theorem}
The topological space $(\Comp,\metric_e)$:
\begin{alphaList}
\item is not Hausdorff.
\item is strongly locally compact (and thus locally compact), but not compact.
\item is locally path-connected (and thus locally connected), but not connected or path-connected.
\item contains no isolated points.
\end{alphaList}
\end{theorem}

\section{Comparison of $\Linf(T)$ and $H^\infty(\D)$ Through Composition Operators}\label{Section:Comparison}

In this section, we wish to compare the spaces $\Linf(T)$ and $H^\infty(\D)$.  Since the norms of both spaces are of the same form operator-theoretic results also follow the same form, especially those pertaining to the composition operator $C_\varphi$.

\begin{compbox}
For function $f:T \to \C$, the supremum-norm of $f$ is \[\|f\|_\infty = \sup_{v \in T} |f(v)|.\]
\tcblower
For $f:\D\to\C$ analytic, the supremum-norm of $f$ is \[\|f\|_\infty = \sup_{z \in \D} |f(z)|.\]
\end{compbox}

\noindent A straightforward calculation shows $\|C_\varphi:H^\infty\to H^\infty\| = 1$ for all analytic self-map $\varphi$ of $\D$, which is what $S(\D)$ refers to in the context of $H^\infty$.  Furthermore, $C_\varphi$ is compact on $H^\infty$ if and only if $\|\varphi\|_\infty < 1$ (see \cite[problem 3.2.2]{CowenMacCluer:1995}).  In \cite{Zheng:2002}, Zheng showed the essential norm of any composition operator on $H^\infty$ is either 0 or 1.  For $C_\varphi$ acting on $\Linf(T)$, compactness is characterized by the range of $\varphi$.  If we define the notation $|\cdot|_\infty: S(T) \to [0,\infty]$ by \[|\varphi|_\infty = \sup_{v \in T} |\varphi(v)|,\] then a self-map $\varphi$ of $T$ has finite range if and only if $|\varphi|_\infty < \infty.$   

\begin{compbox}
For $\varphi \in S(T)$, the composition operator $C_\varphi:\Linf\to\Linf$ is
\begin{alphaList}[leftmargin=2em]
\item bounded with $\|C_\varphi\| = 1$;
\item compact if and only if $|\varphi|_\infty < \infty$, and \[\|C_\varphi\|_e = \begin{cases}1 & \text{if $|\varphi|_\infty = \infty$,}\\0 & \text{if $|\varphi|_\infty < \infty$.}\end{cases}\] 
\end{alphaList}
\tcblower
For $\varphi \in S(\D)$, the composition operator $C_\varphi:H^\infty\to H^\infty$ is
\begin{alphaList}[leftmargin=2em]
\item bounded with $\|C_\varphi\| = 1$;
\item compact if and only if $\|\varphi\|_\infty < 1$, and \[\|C_\varphi\|_e = \begin{cases}1 & \text{if $\|\varphi\|_\infty = 1$,}\\0 & \text{if $\|\varphi\|_\infty < 1$.}\end{cases}\] 
\end{alphaList}
\end{compbox}

While the unit disk $\D$ is a bounded domain for functions in $H^\infty$, and $T$ is an unbounded domain, the conditions can be reconciled in the following manner.  If one considers $\D$ with the hyperbolic metric rather than the Euclidean metric, $T$ can be embedded into $\D$ where $o$ maps to the origin. Such an embedding was studied for $T$ a homogeneous rooted infinite tree by Cohen and Colonna in \cite{CohenColonna:1994}. In this way, the compactness of $C_\varphi$ on $\Linf$ matches those for $C_\varphi$ acting on $H^\infty$.

Composition differences $C_\varphi-C_\psi$ are central in the study of component structures and isolated points.  MacCluer, Ohno, and Zhao characterized the compact difference $C_\varphi-C_\psi$ on $H^\infty$ in \cite[Theorem 3]{MacCluerOhnoZhao:2001} in terms of the pseudo-hyperbolic distance between $z$ and $w$ in $\D$ defined by \[\beta(z,w) = \left|\frac{z-w}{1-\overline{z}w}\right|.\]  

\begin{compbox}
Let $\varphi, \psi \in S(T)$ with $\varphi\neq\psi$. Then the following are equivalent:
\begin{alphaList}[leftmargin=2em]
\item $C_\varphi-C_\psi:\Linf\to\Linf$ is\\ compact;
\item one of \normalfont{(i)} or \normalfont{(ii)} holds:
\begin{enumerate}[leftmargin=1.5em,label={\normalfont{(\roman*)}}]
\item $|\varphi|_\infty,|\psi|_\infty < \infty$,
\item $|\varphi|_\infty = |\psi|_\infty = \infty$ and \[N_{\varphi,\psi} \text{ is finite}.\]
\end{enumerate}
\end{alphaList}
\tcblower
Let $\varphi, \psi \in S(\D)$ with $\varphi\neq\psi$. Then the following are equivalent:
\begin{alphaList}[leftmargin=2em]
\item $C_\varphi-C_\psi:H^\infty\to H^\infty$ is\\ compact;
\item one of \normalfont{(i)} or \normalfont{(ii)} holds:
\begin{enumerate}[leftmargin=1.5em,label={\normalfont{(\roman*)}}]
\item $\|\varphi\|_\infty,\|\psi\|_\infty < 1$,
\item $\|\varphi\|_\infty = \|\psi\|_\infty = 1$ and 
\[\begin{aligned}
0 &= \lim_{|\varphi(z)|\to 1} \beta(\varphi(z),\psi(z))\\
&= \lim_{|\psi(z)|\to 1} \beta(\varphi(z),\psi(z)).
\end{aligned}\]
\end{enumerate}
\end{alphaList}
\end{compbox}

These two results are quite different in nature upon close inspection in the case when neither $C_\varphi$ or $C_\psi$ are compact.  For $C_\varphi-C_\psi$ acting on $\Linf(T)$, the symbols $\varphi$ and $\psi$ must be equal off a finite set.  If one defines the boundary of $T$ in some topological sense, it will need to be the case that $\varphi$ and $\psi$ have the same boundary behavior.  While this is essentially the analogous condition for $C_\varphi-C_\psi$ compact on $H^\infty$, the symbols are not so restricted inside $\D$.  In fact, \cite[Example 1]{MacCluerOhnoZhao:2001} shows the following self-maps of $\D$ \[\begin{aligned}\varphi(z) &= sz+1-s, \quad 0 < s < 1;\\
\psi(z) &= \varphi(z) + t(z-1)^b,\end{aligned}\]
where $b > 2$, $t \in \R$, and $|t|$ small enough so that $\psi$ maps $\D$ into $\D$, induce a compact $C_\varphi-C_\psi$.  These maps are such that $\|\varphi\|_\infty = \|\psi\|_\infty = 1$, but $\varphi(z) \neq \psi(z)$ for all $z \in \D$.

As the compactness characterizations of $C_\varphi-C_\psi$ show fundamental differences between $\Linf(T)$ and $H^\infty(\D)$, so do their norms.  This difference provides the reason why the topological structure of $\Comp$ and $\CC(H^\infty)$ are fundamentally different in the operator norm topology.  MacCluer, Ohno, and Zhao showed that the topological space $\CC(H^\infty)$ with the operator norm topology is homeomorphic to the metric space of analytic self-maps of $\D$ induced by the metric \[d_\beta(\varphi,\psi) = \sup_{z \in \D} \beta(\varphi(z),\psi(z)).\]  

\begin{compbox}
Let $\varphi, \psi \in S(T)$ with $\varphi\neq\psi$. Then for $C_\varphi-C_\psi:\Linf \to \Linf$,\\[10pt] \centerline{$\|C_\varphi-C_\psi\| = 2.$}
\tcblower
Let $\varphi,\psi \in S(\D)$ with $\varphi\neq\psi$. Then for $C_\varphi-C_\psi:H^\infty \to H^\infty$, \[\|C_\varphi-C_\psi\| = \frac{2-2\sqrt{1-d_\beta(\varphi,\psi)^2}}{d_\beta(\varphi,\psi)}.\] 
\end{compbox}

For the case of $\Linf(T)$, the operator norm of $C_\varphi-C_\psi$ makes $\Comp$ into a discrete metric space (under the operator norm topology).  Thus every composition operator is isolated in $\Comp$.  This is very extreme, and very different from the case of $\CC(H^\infty)$.  As noted in \cite{MacCluerOhnoZhao:2001}, the norm of $C_\varphi-C_\psi$ is equal to 2 exactly when $d_\beta(\varphi,\psi) = 1$, and less than 2 otherwise. In fact, \cite[Theorem 2]{MacCluerOhnoZhao:2001} proves this is precisely the characterization of the isolated $C_\varphi:H^\infty \to H^\infty$.  One can see the results on $\Linf$ and $H^\infty$ match when written as follows.

\begin{compbox}
Let $\varphi \in S(T)$. Then the following are equivalent:
\begin{alphaList}[leftmargin=2em]
\item $C_\varphi$ is isolated in $\Comp$.
\item for any $\psi \in S(T)$, $\psi\neq\varphi$, \[\|C_\varphi-C_\psi\|=2.\]
\end{alphaList}
\tcblower
Let $\varphi \in S(\D)$. Then the following are equivalent:
\begin{alphaList}[leftmargin=2em]
\item $C_\varphi$ is isolated in $\CC(H^\infty)$.
\item for any $\psi \in S(\D)$, $\psi\neq\varphi$, \[\|C_\varphi-C_\psi\|=2.\]
\item for any $\psi \in S(\D)$, $\psi \neq \varphi$, \[\sup_{z \in \D}\beta(\varphi(z),\psi(z)) = 1.\]
\end{alphaList}
\end{compbox}
\noindent As $T$ has very little structure, it is an open question whether other discrete objects, with more structure, would yield a topology on $\Comp$ that more closely resembles that of $\CC(H^\infty)$.

We end this section with open questions that would continue this line of research:
\begin{numList}
\item Is $(\Comp,\metric_e)$ $\sigma$-compact, that is can $(\Comp,\metric_e)$ be written as a countable union of compact sets?
\item What is the topological structure of $\mathscr{W}(\Linf)$, the space of weighted composition operators $W_{\psi,\varphi}$?
\item Are there discrete function spaces for which $(\CC,\metric_u)$ is not totally disconnected?
\item Are there discrete function spaces for which $\|C_\varphi-C_\psi\|$ or $\|C_\varphi-C_\psi\|_e$ is not a simple function?  
\end{numList}

\bibliographystyle{amsplain}
\bibliography{references.bib}
\end{document}